\documentclass{amsart}
\usepackage{amsmath}
\usepackage{amssymb}
\usepackage{amsfonts}

\usepackage[dvips]{graphicx,color}
\usepackage{amsthm}
\usepackage{amscd}
\usepackage{amsfonts}
\usepackage[all]{xy}

\DeclareMathOperator{\Out}{Out}
\DeclareMathOperator{\Hom}{Hom}
\DeclareMathOperator{\Ima}{Im}
\DeclareMathOperator{\Ker}{Ker}
\DeclareMathOperator{\Der}{Der}
\DeclareMathOperator{\ODer}{ODer}
\DeclareMathOperator{\IDer}{IDer}
\DeclareMathOperator{\Aut}{Aut}
\DeclareMathOperator{\Inn}{Inn}

\DeclareMathOperator{\grad}{gr}
\DeclareMathOperator{\Diff}{Diff}
\DeclareMathOperator{\Sp}{Sp}
\DeclareMathOperator{\GL}{GL}
\DeclareMathOperator{\IOut}{IOut}
\DeclareMathOperator{\ad}{ad}
\DeclareMathOperator{\Ad}{Ad}
\DeclareMathOperator{\Isom}{Isom}
\DeclareMathOperator{\IAut}{IAut}

\def\Z{{\mathbb{Z}}}

\def\R{{\mathbb{R}}}

\def\th{{\theta}}

\def\0{{{\{ 0 \} }}}

\def\w{{\wedge}}

\def\>{{{\geq}}}
\def\<{{{\leq}}}

\def\Sp{{\rm Sp}}

\def\A{{\mathcal{A}}}

\def\gr{{\pi}}
\def\ct{{\hat{L}H}}
\def\cR{{\hat{\R}}}
\def\Th{{\bar{\Theta}}}
\def\ia{{\IOut(\ct /I)}}
\def\de{{\ODer^+(\ct /I)}}

\def\ch{{\Phi_E}}

\theoremstyle{definition}
  \newtheorem{thm}{Theorem}[section]
  
  \newtheorem{lem}[thm]{Lemma}

  \newtheorem{prop}[thm]{Proposition}

\theoremstyle{definition}
  \newtheorem{defi}{Definition}[section]

  \newtheorem{rem}[thm]{Remark}
  \newtheorem{ex}{Example}

\title{Characteristic classes of fiber bundles}

\author{Takahiro Matsuyuki}
\address{Department of Mathematics, 
Tokyo Institute of Technology, 
2-12-1 Oh-okayama, Meguro-ku, Tokyo 152-8551, Japan.}
\email{matsuyuki.t.aa@m.titech.ac.jp}

\author{Yuji Terashima}
\address{Department of Mathematical and Computing Sciences,
Tokyo Institute of Technology,
2-12-1 Ookayama, Meguro-ku, Tokyo 152-8552, Japan.}

\email{tera@is.titech.ac.jp}

\date{}
\begin{document}
\maketitle

\begin{abstract}
In this paper, we construct new characteristic classes of fiber bundles 
via flat connections with values in  
infinite-dimensional Lie algberas of derivations. 
In fact, choosing a fiberwise metric, 
we construct a chain map to the de Rham complex 
on the base space, and show that the induced map 
on cohomology groups is independent of the choices.
Moreover, we show that applying the construction to a surface bundle, our construction gives Morita-Miller-Mumford classes. 
\end{abstract}

\section{Introduction}
A purpose of this paper is to construct characteristic 
classes of fiber bundles which are not necessarily 
principal bundles whose structure group is 
a finite-dimensional Lie group. The difficulty is that a
 diffeomorphism group, which is considered as 
the structure group for a general fiber bundle, 
is very huge compared to a finite-dimensional 
Lie group. 

An idea to overcome the difficulty is a 
`linearization' which means to replace the 
diffeomorphism group ${\rm Diff}(X)$ for the fiber $X$ 
with the 
automorphism group of the tensor algebra 
of the first homology group $H=H_1 (X)$. A main tool is the Maurer-Cartan form of
 the space of expansions which was originally considered
 by Kawazumi \cite{Kaw1, Kaw2} in the case of free groups.

Diagrammatically, the construction is as follows: 

\begin{center}
\includegraphics[width=13cm]{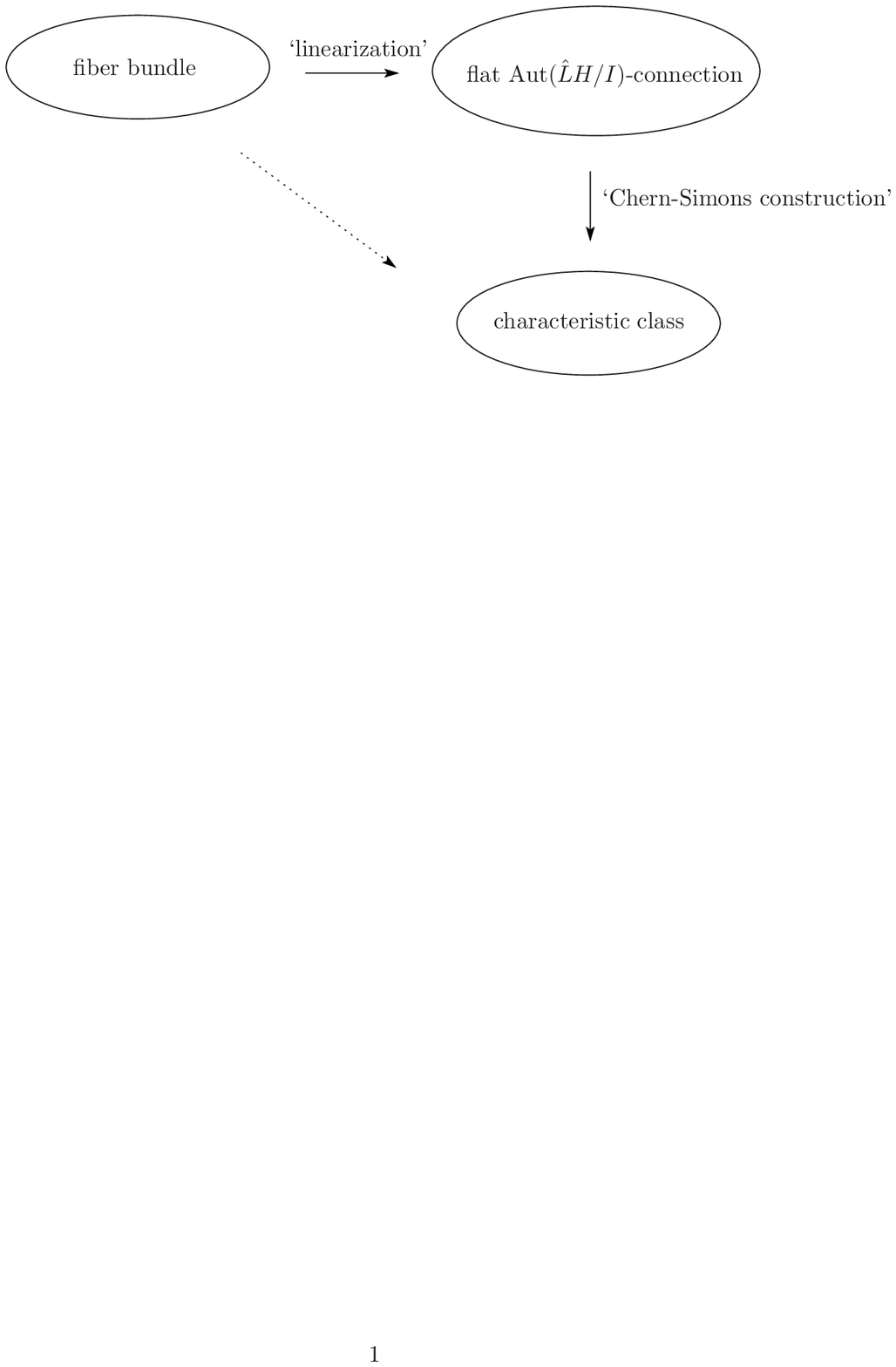}
\end{center}

In this paper, for any fiber bundle $E \to B$ whose structure group satisfies a certain condition, choosing a fiberwise metric, we construct a chain map from 
the Chevalley-Eilenberg complex to the de Rham complex 
on the base space $B$ via  
a flat $\Aut ( \ct /I)$-connection, and show that the induced map 
on cohomology groups is independent of the choices.
Moreover, we show that applying the construction to a closed surface bundle, our construction 
gives Morita-Miller-Mumford classes. The similar construction of Morita-Miller-Mumford classes of 1-punctured surface bundles was previously given in \cite{Kaw2}. 

It is interesting to compare our construction to a different approach on diffeomorphism groups with noncommutative geometry in Lott \cite{L}.

The paper is organized as follows: in Section 2, we introduce tools for construction of characteristic classes. The notions defined in Subsection 2.1 and 2.2 is used in Section 3, and Johnson maps defined in Subsection 2.3 is in Section 4. In Section 3.1 and 3.2, we construct characteristic classes under different two conditions of fiber bundles. In Subsection 3.3, we clarify the relation between characteristic classes constructed in Subsection 3.1 and 3.2. In Subsection 3.4, we describe these characteristic classes by Lie algebra cohomologies of Lie algebra of derivations. In Section 4, we prove that our characteristic classes of a closed surface bundle give Morita-Miller-Mumford classes. 
\medskip

\noindent
{\bf Acknowledgment.} The authors would like to thank T.~Gocho, H.~Kajiura, 
H.~Kasuya, A.~Kato, T.~Sakasai, and T.~Satoh for helpful conversation. Y.~T. is partly supported by Grants-in-Aid for Scientific Research. 
\section{Preliminaries}
For a finitely generated group $\gr$, we set $H=\gr^{\rm ab} \otimes \R$. We consider the completed tensor Hopf algebra $\hat{T}H$ and the completed free Lie algebra $\hat{L}H$ generated by $H$. The Lie algebra $\hat{L}H$ can be regarded as the primitive part of the completed Hopf algebra $\hat{T}H$. (For completed Hopf algebra, see \cite{Q}.) Let $J$ be the augmented ideal of $\hat{T}H$, which consists of all elements whose constant term is zero. The algebra $\hat{T}H$ is described by the projective limit
of finite-dimensional vector spaces 
\[\hat{T}H=\prod_{n=0}^\infty H^{\otimes n}=\varprojlim_n\hat{T}H/J^n,\]
so it is endowed with the limit topology. We define $L_{\geq k}:=\hat{L}H\cap J^k$ and $L_k:=\hat{L}H\cap H^{\otimes k}$.

A closed two-sided Lie ideal $I$ contained in $L_{\geq2}$ is called a {\bf decomposable ideal} of $\hat{L}H$. We identify $I$ with the closed two-sided ideal generated by $I$ in $\hat{T}H$. Then $\hat{T}H/I$ is a completed Hopf algebra. 

We denote the completion of the group Hopf algebra $\R\pi$ of $\pi$ by $\hat{\R}\pi$.
\subsection{Automorphisms and derivations}
In this subsection, we will define infinite-dimensional Lie groups satisfying the diagram
\[\xymatrix{1\ar[r]&\Inn (\ct/I)\ar[r]&\IAut(\ct/I)\ar[r]&\ar[r]\IOut(\ct/I)&1\\
1\ar[r]&\Inn (\ct)\ar[u]^{\text{surj.}}\ar[r]&\IAut_I(\ct)\ar[u]^{\text{surj.}}\ar[r]&\ar[r]\IOut_I(\ct)
\ar[u]^{\text{surj.}}&1}\]
and corresponding Lie algebras 
\[\xymatrix{0\ar[r]&\IDer (\ct/I)\ar[r]&\Der^+(\ct/I)\ar[r]&\ar[r]\ODer^+(\ct/I)&0\\
0\ar[r]&\IDer (\ct)\ar[u]^{\text{surj.}}\ar[r]&\Der_I^+(\ct)\ar[u]^{\text{surj.}}\ar[r]&\ar[r]\ODer^+_I(\ct)
\ar[u]^{\text{surj.}}&0}\]
for a decomposable ideal $I$. Here the rows are exact.

We define the group of {\bf positive automorphisms} of $\ct/I$ by 
\[\IAut(\ct/I):=\{f\in \Aut(\ct/I);f(x)=x+L_{\geq 2}/I\},\]
where $\Aut(\ct/I)$ is the group of completed Lie algebra automorphisms of $\ct/I$. It is a projective limit of finite-dimensional Lie groups, and has the Lie algebra 
\[\Der^+(\ct/I):=\{X\in {\rm Der}(\ct/I);X(H)\subset L_{\geq2}/I\},\]
where ${\rm Der}(\ct/I)$ is the Lie algebra consisted of continuous Lie algebra derivations on $\ct/I$. We call an element of $\Der^+(\ct/I)$ {\bf positive derivation} on $\ct/I$. Since Hopf algebra automorphisms and Hopf derivations of $\hat{T}H/I$ are determined by their restrictions to the primitive part $\ct/I$, $\IAut(\ct/I)$ and $\Der^+(\ct/I)$ can be regarded as a subgroup of completed Hopf algebra automorphisms of $\hat{T}H/I$ and a Lie subalgebra of continuous Hopf derivations on $\hat{T}H/I$, respectively. Here Hopf derivation means derivation and coderivation. 

An inner automorphism of $\ct/I$ is a Hopf algebra automorphism $\iota_a:\ct/I\to \ct/I$ defined for all $x$ in $\ct/I$ by 
\[\iota_a(x):=axa^{-1}\]
where $a$ is a given fixed group-like element of $\hat{T}H/I$. Then $\iota_a$ is a positive automorphism of $\ct/I$. We denote the normal subgroup of $\IAut(\ct/I)$ consisting of inner automorphisms by $\Inn(\ct/I)$. It has the Lie algebra 
\[\IDer(\ct/I):=\{\ad(a)\in \Der^+(\ct/I);a\in\ct/I\},\]
which is the Lie algebra of inner derivations on $\ct/I$. We also define the quotient group
\[\IOut(\ct/I):=\IAut(\ct/I)/\Inn(\ct/I)\]
and its Lie algebra
\[\ODer^+(\ct/I):=\Der^+(\ct/I)/\IDer(\ct/I).\]

We also use the groups of positive automorphisms and outer automorphisms of $\ct$ which preserve a decomposable ideal $I$
\[\IAut_I(\ct):=\{f\in \IAut(\ct);f(I)=I\},\quad \IOut_I(\ct):=\IAut_I(\ct)/\Inn(\ct)\]
These Lie algebras are the Lie algebras of positive derivations and outer derivations on $\ct$ which preserve $I$
\[\Der^+_I(\ct):=\{X\in \Der^+(\ct);X(I)\subset I\},\quad \ODer^+_I(\ct):=\Der^+_I(\ct)/\IDer(\ct), \]
respectively.

If $I$ is a homogeneous ideal, the Lie algebra of derivations on $\hat{L}H/I$ has the natural grading.
\begin{defi}
Let $H$ be a vector space and $I$ a homogeneous decomposable ideal of $\hat{L}H$. The degree $k$ component of $\Der(\ct/I)$ is defined by
\[\Der^k(LH/I):=\{X\in \Der^+(\ct/I);X(H)\subset L_{k+1}/(L_{k+1}\cap I)\}.\]
Then we have the decomposition
\[\Der^+(\ct/I)=\prod_{k=1}^\infty\Der^k(LH/I).\]
In the same way, $\IDer^k(LH/I)$ and $\ODer^k(LH/I)$ is defined.
\end{defi}
\subsection{Space of expansions}

In this subsection, we define the spaces containing Chen expansions of a fundamental group. These spaces play an important role in our construction of characteristic classes in Section 3. It is a genaralization of the spaces of Magnus expansions of free groups in \cite{Kaw2}.

\begin{defi}
For a decomposable ideal $I$ of $\hat{L}H$, an {\bf $I$-expansion} of $\gr$ is a completed Hopf algebra isomorphism $\th : 
\cR \gr \to \hat{T}H /I$ satisfying 
\[\th (g)=1+[g]+J^2/I\]
for any $g \in \gr$. The space of $I$-expansions is denoted by $\Theta (\gr ,I)$. 
\end{defi}

The group $\IAut(\ct/I)$ acts on $\Theta (\gr ,I)$ as follows:
\[f\cdot \theta:=f\circ \theta\]
for $f\in\IAut(\ct/I)$ and $\theta\in\Theta(\gr,I)$. This action is free and transitive because $\theta\circ\theta_0^{-1}\in\IAut(\ct/I)$ and 
\[(\theta\circ\theta_0^{-1})\cdot\theta_0=\theta\]
for all $\theta$, $\theta_0\in \Theta(\gr,I)$. We describe the space of conjugacy classes of $I$-expansion by $\Th(\gr,I):=\Inn(\ct/I)\backslash \Theta(\gr,I)$. The group $\ia$ acts on $\Th(\gr,I)$ freely and transitively. (We denote the conjugacy class of $\theta\in\Theta(\gr,I)$ by $[\theta]$.)

We consider the space $\Th(\gr)$ of conjugacy classes of all possible expansion of $\gr$. It is parametrized by \[\IOut(\ct)\times_{\IOut_{I_0}(\ct)}\Th(\gr,I_0)\] fixing a decomposable ideal $I_0$ which satisfies $\Theta(\gr,I_0)\neq\emptyset$. Here we consider $\IOut_I(\ct)$ acts on $\Th(\gr,I)$ through the natural homomorphism $\IOut_I(\ct)\to \ia$. The automorphism group $\Out(\pi)$ of $\pi$ acts on $\Th(\pi)$ by
\[\varphi\cdot [\theta]=[|\varphi|\circ \theta\circ \varphi^{-1}]\]
for $\varphi\in\Out(\pi)$, $\theta\in \Th(\pi)$. Here $|\varphi|$ means the Hopf algebra isomorphism between quotient Hopf algebras of the tensor Hopf algebra $\hat{T}H$ induced by $\varphi\in\Out(\gr)$. 

The space $\mathcal{I}(\pi)$ of all ideal $I$ satisfying $\Th(\pi,I)\neq\emptyset$ is parametrized by \[\IOut(\ct)/\IOut_{I_0}(\ct)=\IAut(\ct)/\IAut_{I_0}(\ct)\] fixing  a decomposable ideal $I_0$ which satisfies $\Th(\gr,I_0)\neq\emptyset$. The group 
\[\GL(\pi):=\Ima (\Aut(\pi)\to \GL(H))=\Ima (\Out(\pi)\to \GL(H))\]
acts on $\mathcal{I}(\pi)$ by 
\[\varphi\cdot I=\varphi(I)\]
for $\varphi\in \GL(\pi)$, $I\in\mathcal{I}(\pi)$. Through the natural homomorphism $\Out(\pi)\to \GL(\pi)$, the natural map $\Th(\pi)\to\mathcal{I}(\pi)$ is regarded as an $\Out(\pi)$-equivariant $\Th(\pi,I_0)$-bundle.
 
\subsection{Generalization of a Johnson map} 
We define a generalization of a Johnson map for free groups by \cite{Kaw1}. The results in this subsection shall be used only in Section 4.

\begin{prop}\label{prop:unique}
{\it Let $\pi$ be a finitely generated group. If $\gr$ has a homogeneous decomposable ideal, i.e., there exists a homogeneous decomposable ideal $I$ such that $\Theta(\pi,I)\neq \emptyset$, then such an ideal is unique. We denote it by $I_\pi$.}  
\end{prop}
\begin{proof}
Suppose a homogeneous decomposable ideal $I'$ also satisfies $\Theta(\pi,I')\neq\emptyset$. Take $\theta\in\Theta(\pi,I)$, $\theta'\in\Theta(\pi,I')$ and set $f:=\theta'\circ\theta^{-1}$. Then 
\[f(x)=x+J^2/I'\]
for $x\in H$. From the equation
\[0=f(y)=y+J^{k+1}/(J^{k+1}\cap I')\]
for $y\in H^{\otimes k}\cap I$, we get $y\in I'$ by assumption that $I'$ is homogeneous. Since $I$ is also homogeneous, we obtain $I\subset I'$. In the same way, we can prove $I\supset I'$. Therefore we have $I=I'$. 
\end{proof}
\begin{defi}Let $\gr$ be a finitely generated group which has a homogeneous decomposable ideal $I_\gr$. 
For $\theta\in \Theta(\gr,I_\gr)$, we define the {\bf Johnson map} $\tau^\theta:\Aut(\gr)\to \IAut(\ct/I_\gr)$ associated to $\theta$ by
\[\tau^\theta(\varphi)=\theta\circ \varphi\circ\theta^{-1}\circ|\varphi|^{-1}.\]
Note that uniqueness of $I_\pi$ gives $|\varphi|(I_\pi)=I_\pi$. By the decomposition 
\[\IAut(\ct/I_\pi)\subset \prod_{p=1}^\infty\Hom(H,L_{p+1}/(L_{p+1}\cap I_\pi)),\]
we denote the $\Hom(H,L_{p+1}/(L_{p+1}\cap I_\pi))$ component of $\tau^\theta$ by $\tau^\theta_p$. 
\end{defi}
For a group-like element $a\in \hat{T}H/I_\gr$, the equation
\[\tau^{\iota_a\theta}(\varphi)=\iota_a\theta\varphi\theta^{-1}\iota_a^{-1}|\varphi|^{-1}=\iota_{a}\tau^\theta(\varphi)\iota_{|\varphi|(a)}^{-1}\]
holds. Then we can define $\tau^{[\theta]}:\Out(\gr)\to \IOut(\ct/I_\gr)$ for $[\theta]\in\Th(\gr,I_\gr)$. We can also obtain $\tau_1^{[\theta]}$ for $[\theta]\in\Th(\gr,I_\gr)$.

\begin{defi}Given a filtered algebra $A$ with decreasing filtration $\{F_k(A)\}_{k=0}^\infty$,
\[\grad(A):=\bigoplus_{k=0}^\infty F_k(A)/F_{k+1}(A)\]
is a graded algebra by the multiplication induced by the multiplication of $A$. When $A$ is the completion of a graded algebra $V$, $\gr(\hat{V})=V$ holds.

Similar thing happens when $G$ is a group with decreasing central filtration $\{F_k(G)\}_{k=0}^\infty$. Then
\[\grad(G):=\bigoplus_{k=0}^\infty F_k(G)/F_{k+1}(G)\]
is a graded Lie algebra by the Lie bracket induced by the commutator of $G$.
\end{defi}

The following Propositions \ref{prop:emb}, \ref{prop:John}, \ref{prop:aut} and \ref{prop:out} are straightforward generalizations of those in \cite{Kaw1}.
\begin{prop}\label{prop:emb}
{\it For an $I_\gr$-expansion $\theta\in\Theta(\gr,I_\gr)$, the map $\grad(\theta):\grad(\gr)\otimes \R\to \grad(\hat{T}H/I_\gr)=TH/I_\gr$ induced by $\theta$ is the natural identification $\grad(\gr)\otimes \R\to LH/I_\gr$. Especially, $\grad(\theta)$ does not depend on a choice of an $I_\gr$-expansion $\theta$. (Here the group $\gr$ and the algebra $\hat{T}H$ is filtered by the lower central series $\{\Gamma_k\gr\}_{k=1}^{\infty}$ and the power series $\{J^k\}_{k=1}^\infty$ of the argumented ideal $J$, respectively.) }
\end{prop}
\begin{proof}
Decompose $\theta$ into the sum 
\[\theta=\sum_{k=0}^\infty\theta_k\]
with respect to grading. For a positive integer $k$, $x\in \Gamma_{k-1}\gr$ and $y\in \gr$, 
\[\theta([x,y])=1+\theta_{k-1}(x)[y]-[y]\theta_{k-1}(x)+J^{k+1}/(J^{k+1}\cap I_\gr).\]
Therefore $\theta_k([x,y])=[\theta_{k-1}(x),[y]]$ holds. Thus the result follows from it by induction.
\end{proof}

Let $\gr$ be a finitely generated group. The {\bf Andreadakis filtration} of the automorphism group $\Aut(\gr)$ of $\gr$ is defined by
\[\mathcal{A}_\gr(k):=\Ker(\Aut(\gr)\to \Aut(\gr/\Gamma_{k+1}\gr)).\]
Remark $\mathcal{A}_\gr(1)=\IAut(\gr)$. The $k$-th Johnson morphism $\tau_k:\mathcal{A}_\gr(k)/\mathcal{A}_\gr(k+1)\to \Hom(\gr^{\rm ab},\grad^{k+1}(\gr))$ is defined as follows: For $\varphi\in \mathcal{A}_\gr(k)$ and $x\in \gr$, we set $s_\varphi(x):=x^{-1}\varphi(x)\in \Gamma_{k+1}\gr$. If $x\in \Gamma_2\gr$, $s_\varphi(x)\in \Gamma_{k+2}\gr$ holds. Thus, we can define $\tau_k(\varphi)([x])=[s_\varphi(x)]\in \Gamma_{k+1}\gr/\Gamma_{k+2}\gr$. In addition, the fact $\tau_k(\varphi)=0$ for $\varphi\in \mathcal{A}_\gr(k+1)$ induces the homomorphism
\[\tau_k:\mathcal{A}_\gr(k)/\mathcal{A}_\gr(k+1)\to \Hom(\gr^{\rm ab},\grad^{k+1}(\gr)).\]
The direct sum of these maps with respect to $k$ defines the Lie algebra homomorphism
\[\tau:\grad^+(\Aut(\gr))\to \Der^+(\grad(\gr))\]
(see \cite{S} for detail). Here we denote the positive degree part of $\grad(G)$ by $\grad^+(G)$ for a group $G$ with central filtration. By definition, an inner automorphism is mapped to an inner derivation by this map. Then the Lie algebra homomorphism
\[\bar{\tau}:\grad^+(\Out(\gr))\to \ODer^+(\grad(\gr))\]
is induced by $\tau$. Here $\Out(\gr)$ is filtered by
\[\mathcal{O}_\gr(k):=\Ker(\Out(\gr)\to \Out(\gr/\Gamma_{k+1}\gr)).\]
\begin{prop}\label{prop:John}
{\it For $\theta\in \Theta(\gr,I_\gr)$ and $m\geq1$, the map induced by $\tau^\theta$
\[\grad^m(\tau^\theta):\grad^m(\Aut(\gr))\to \grad^m(\Aut(\ct/I_\gr))=\Der^m(LH/I_\gr)\]
equals the $m$-th Johnson homomorphism $\tau_m$ through the natural identification $\grad(\gr)\otimes\R\to LH/I_\gr$. Especially $\grad^+(\tau^\theta)$ does not depend on a choice of an $I_\gr$-expansion $\theta$. Here the groups $\Aut(\gr)$ and $\Aut(\ct/I_\gr)$ is filtered by $\{\mathcal{A}_\gr(k)\}_{k=0}^{\infty}$ and $\{\Aut^{\geq k}(\ct/I_\gr)\}_{k=0}^\infty$, where \[\Aut^{\geq k}(\ct/I_\gr)=\{f\in \Aut(\ct/I_\gr);f(x)=x+ L_{\geq k+1}/(L_{\geq k+1}\cap I_\gr)\text{ for } x\in H\},\]respectively. }
\end{prop}

\begin{proof}
Suppose $\tau^\theta_k(\varphi)=\tau_k(\varphi)$ for $\varphi\in \A_\gr(k)$ and $1\leq k< m$. Then $\tau^\theta_k(\varphi)=0$ holds for $\varphi\in \A_\gr(m)$ and $1\leq k< m$. From
\[\tau^\theta(\varphi)([x])=\theta\varphi\theta^{-1}([x])=\theta(\varphi(x))+\tau_m^\theta(\varphi)([x])\]
for $x\in\gr$, we obtain
\[\tau_m(\varphi)([x])=\theta_{m+1}(s_\varphi(x))\equiv \theta(s_\varphi(x))-1=\theta(x)^{-1}\theta(\varphi(x))-1
\equiv\tau_m^\theta(\varphi)([x])\]
modulo $L_{\geq m+2}/(L_{\geq m+2}\cap I_\gr)$. Then we have $\tau^\theta_m(\varphi)=\tau_m(\varphi)$ for all $m\geq 1$ and $\varphi\in \A_\gr(m)$ inductively. 
\end{proof}

\begin{prop}\label{prop:aut}
{\it The first Johnson map $\tau^\theta_1$ is a cocycle of $\Aut(\gr)$ with coefficient $\Der^1(LH/I_\gr)$, and the cohomology class $[\tau^\theta_1]\in H^1(\Aut(\gr); \Der^1(LH/I_\gr))$ does not depend on a choice of an $I_\gr$-expansion $\theta$. }
\end{prop}
\begin{proof}
The cocycle condition of $\tau^\theta_1$ is written by
\[\tau^\theta_1(\varphi\psi)=\tau^\theta_1(\varphi)+|\varphi|\tau^\theta_1(\psi)\]
for every $\varphi,\psi\in \Aut(\gr)$. This follows from the formula \[\tau^\theta(\varphi\psi)=\tau^\theta(\varphi)|\varphi|\tau^\theta(\psi)|\varphi|^{-1}.\]

Independency of its cohomology class is proved as follows. Because $\theta_2'(x)=\theta_2(x)$ for $x\in \Gamma_2\gr$ by Prop \ref{prop:emb}, we can define $F\in \Der^1(LH/I_\gr)$ by
\[F([x]):=\theta_2'(x)-\theta_2(x)\]
for $[x]\in H$. By the formula
\[\tau^{\theta}_1(\varphi)([x])=\theta_2(x)-|\varphi|\theta_2(\varphi^{-1}(x))\]
for $[x]\in H$ and $\varphi\in \Aut(\gr)$, we have
\[\tau^{\theta}_1-\tau^{\theta'}_1=dF.\]
Thus $[\tau^{\theta'}_1]=[\tau^\theta_1]\in H^1(\Aut(\gr);\Der^1(LH/I_\gr))$.
\end{proof}

Similar results for the outer isomorphism group also holds:
\begin{prop}\label{prop:out}
{\it For $m\geq1$, the map induced by $\tau^\theta$
\[\grad^m(\tau^\theta):\grad^m(\Out(\gr))\to \grad^m(\Out(\ct/I_\gr))=\ODer^m(LH/I_\gr)\]
for $\theta\in\Th(\gr,I_\gr)$ is identified with the $m$-th Johnson homomorphism $\bar{\tau}_m$ of $\Out(\gr)$. Here the groups $\Out(\gr)$ and $\Out(\ct/I_\gr)$ is filtered by $\{\mathcal{O}_\gr(k)\}_{k=0}^{\infty}$ and $\{\Out^{\geq k}(\ct/I_\gr)\}_{k=0}^\infty$, where \[\Out^{\geq k}(\ct/I_\gr)=\Ima(\Aut^{\geq k}(\ct/I_\gr)\to \Out(\ct/I_\gr)),\]respectively. Furthermore the first Johnson map $\tau^\theta_1$ is a cocycle of $\Out(\gr)$ with coefficient $\ODer^1(LH/I_\gr)$, and the cohomology class \[[\tau^\theta_1]\in H^1(\Out(\gr); \ODer^1(LH/I_\gr))\] does not depend on a choice of $\theta\in\Th(\gr,I_\gr)$. } 
\end{prop}

\section{Characteristic classes}

For an oriented closed manifold $X$, we set the 
fundamental group $\gr =\pi_1 (X)$ and the first homology group 
$H=H_1 (X;\R )$. For the purpose, we recall the result of K. T. Chen.
\begin{defi}[\cite{C1,C2}]Let $X$ be a manifold. We denote the suspension of  $H_+(X;\R):=\bigoplus_{p>0}H_p(X;\R)$ by $H_+:=H_+(X;\R)[1]$. The completed tensor algebra $T:=\hat{T}H_+$ of the suspension $H_+$ is a (completion of) graded algebra. A pair $(\omega,\delta)$ satisfying the following conditions is a {\bf formal homology connection}:
\begin{enumerate}
\item a $T$-coefficient differential form $\omega\in A^*(X)\otimes T$ is described by 
\[\omega=\sum_{k=1}^{\infty}\sum_{i_1,\dots,i_k}\omega_{i_1\cdots i_k}X_{i_1}\cdots X_{i_k},\]
where $X_1,\dots,X_n$ is a homogeneous basis of $H_+$ and the differential form $\omega_{i_1\cdots i_k}\in A^*(X)$ is a $(\deg X_{i_1}+\cdots+\deg X_{i_k}+1)$-form and they satisfy
\[\int_{X_p}\omega_p=1.\]
\item  a linear map $\delta:T\to T$ is a differential with degree $-1$ of the graded algebra $T$ such that 
\[\delta(H_+)\subset \prod_{q=2}^\infty H_+^{\otimes q}.\]
\item the flatness condition $\delta\omega=d\omega-\epsilon(\omega)\wedge \omega$ holds, where $\epsilon:A^*(X)\to A^*(X)$ is defined by $\epsilon(\alpha)=(-1)^p\alpha$ for $\alpha\in A^p(X)$, and $\delta$ and $\epsilon$ are extended onto $A^*(X)\otimes T$.
\end{enumerate}
\end{defi} 

Let $(\omega,\delta)$ be a formal homology connection of $X$. Then we can obtain the chain map $C_*(\Omega X;\R)\to T$ from the cubical chain complex of the loop space $\Omega X$ of $X$ to the algebra $T$ defined by the iterated integral
\[\sigma\mapsto \sum_{n=0}^\infty\int_\sigma\underbrace{\omega\cdots \omega}_n.\]
Furthermore, by the result of Chen, the homomorphism $H_0(\Omega X;\R)\to H_0(T,\delta)$ induces
a Hopf algebra isomorphism $\theta_\omega:\hat{\R}\gr\simeq \hat{T}H/I_\omega$ where $I_\omega:=\delta(H_2(X;\R)[1])$. It is called Chen expansion. In our notation, $I_\omega$ is a decomposable ideal and $\theta_\omega$ is an $I_\omega$-expansion. 

\begin{thm}[\cite{C1, C2}]
{\it Let $(X,g)$ be an oriented closed Riemannian manifold and 
\[A^*(X)=\mathcal{H}_g\oplus dA^*(X)\oplus d^*_gA^*(X)\]
be the Hodge decomposition of $(X,g)$. Here $\mathcal{H}_g$ is the space of harmonic forms and $d_g^*:A^*(X)\to A^*(X)$ is the adjoint operator of $d$ with respect to $g$. Then there exists a unique formal homology connection $(\omega_g,\delta_g)$ such that 
\[\omega_g=\sum_{i=1}^m\omega_iX_i+\sum_{p\geq 2}\sum_{i_1,\dots,i_p}\omega_{i_1\dots i_p}X_{i_1}\cdots X_{i_p}\]
where $X_1,\dots,X_m$ is a basis of $H_+$, $\omega_i\in \mathcal{H}_g$ and $\omega_{i_1\dots i_p}\in d^*_gA^*(X)$. }
\end{thm}

We denote the group of diffeomorphisms of $X$ preserving the orientation by ${\rm Diff}_+(X)$. For a Riemannian metric $g$ on $X$ and $\varphi\in{\rm Diff}_+(X)$, we define the metric $\varphi_*g$ on $X$ by
\[(\varphi_*g)(u,v):=g(\varphi^*u,\varphi^*v)\]
for cotangent vectors $u,v\in T^*_xX$ and $x\in X$. Then, since 
\[\mathcal{H}_{\varphi_*g}=(\varphi^*)^{-1}(\mathcal{H}_g),\quad d^*_{\varphi_*g}A^*(X)=(\varphi^*)^{-1}(d^*_gA^*(X))\]
for $\varphi\in {\rm Diff}_+(X)$, we have
\[(\omega_{\varphi_*g},\delta_{\varphi_*g})=(((\varphi^*)^{-1}\otimes |\varphi|)(\omega_g),|\varphi|\circ\delta_g\circ|\varphi|^{-1}).\]
Here $|\varphi|$ means the algebra isomorphism $\hat{T}H\to \hat{T}H$ induced by $\varphi\in{\rm Diff}_+(X)$. Thus the corresponding Chen expansions and ideals satisfy
\begin{equation}\theta_{\omega_{\varphi_*g}}=|\varphi|\circ\theta_{\omega_g}\circ \varphi^{-1}_*:\hat{\R}\pi_1(X,\varphi(*))\to \hat{T}H/|\varphi|(I_{\omega_g}),\quad I_{\omega_{\varphi_*g}}=|\varphi|(I_{\omega_g})\label{equiv}\end{equation}
for a diffeomorphism $\varphi$ of $X$ preserving the base point $*$. Here we denote the homomorphism $\hat{T}H/I_{\omega_g}\to \hat{T}H/|\varphi|(I_{\omega_g})$ induced by $|\varphi|:\hat{T}H\to \hat{T}H$ by the same letter $|\varphi|$.

Based on these considerations, we describe characteristic classes of fiber bundles using expansion spaces as follows: let $E\to B$ be an oriented fiber bundle whose fiber $X$ is an oriented closed manifold and set $\pi=\pi_1(X)$. We choose a fiberwise metric $g_{E/B}$ on $E \to B$. Take an open covering $\{U_i\}_i$ of $B$ and local trivializations $\varphi_i:\pi^{-1}(U_i)\simeq U_i\times X$ of $E$. Then $g_{E/B}$ defines fiberwise metrics $g_i$ of trivial bundles $U_i\times X\to U_i$. The map $\theta_i:U_i\to \Th(\gr)$ can be defined by $x\mapsto [\theta_{\omega_{g_i(x)}}]$.
For any differential form $\alpha\in A^*(\Th(\gr))^{\Out(\gr)}$, 
\[\theta_i^*\alpha=\theta_j^*\alpha\]
holds on $U_i\cap U_j$ because the correspondence from metrics to Chen expansions is ${\rm Diff}_+(X)$-equivariant from the equations (\ref{equiv}). (It also means $\theta_i^*\alpha$ is independent of a choice of local trivializations.) Thus the family $\{\theta_i^*\alpha\}_i$ of differential forms on the open covering defines the differential form $\theta^*\alpha$ on $B$ by gluing. The correspondence $\alpha \mapsto \theta^*\alpha$ is a chain map $A^*(\Th(\gr))^{\Out(\gr)}\to A^*(B)$. Given two distinct open coverings, we can prove that the maps $A^*(\Th(\gr))^{\Out(\gr)}\to A^*(B)$ for two coverings are equal by taking a refinement of these open coverings. 

Since any two fiberwise metrics $g_0$, $g_1$ of a fiber bundle can be connected by a segment $(1-t)g_0+tg_1$, the chain maps $A^*(\Th(\gr))^{\Out(\gr)}\to A^*(B)$ for distinct metrics are chain homotopy equivalent. Therefore the induced homomorphism 
\[\ch:H^*_{DR}(\Th(\gr))^{\Out(\gr)}\to H^*_{DR}(B)\]
does not depend on a choice of a fiberwise metric. So the following theorem holds:

\begin{thm}\label{thm:map}
{\it For an oriented fiber bundle $E\to B$ whose fiber is $X$, the map $A^*(\Th(\gr))^{\Out(\gr)}\to A^*(B) $ constructed above is a chain map, and the induced map $\Phi_E$ on cohomology groups is independent of a choice of a fiberwise metric.}
\end{thm}
\begin{rem}\label{rem:map}
Let $\mathcal{G}$ be the structure group of $E\to B$ and we set $\tilde{S}:=\Ima(\mathcal{G}\to \Out(\gr))$. We obtain the map
\[H^*_{DR}(\Th(\gr))^{\tilde{S}}\to  H^*_{DR}(\Th(\gr))^{\Out(\gr)}\overset{\ch}\to H^*_{DR}(B).\]
We also denote this map by $\Phi_E$.
\end{rem}
\subsection{Homologically trivial bundles}
In this subsection we assume that a fiber bundle $E\to B$ with fiber $X$ has the structure group
\[\mathcal{T}(X):=\Ker(\Diff_+(X)\to \GL(H)).\]
Fix a decomposable ideal $I_0$ which satisfies $\Theta(\gr,I_0)\neq\emptyset$. We choose a fiberwise metric $g$ of $E\to B$. Since the structure group of $E\to B$ is $\mathcal{T}(X)$, the correspondence of ideals by $g$ gives a map $q:B\to\mathcal{I}(\gr)$. Because the topological group $\IOut_{I_0}(\ct)$ is contractible, the pullback $q^*\IOut(\ct)\to B$ of the principal $\IOut_{I_0}(\ct)$-bundle $\IOut(\ct)\to \mathcal{I}(\gr)$ is trivial. Taking a trivialization of the principal bundle, we get the $\IOut(\gr)$-equivariant map  
\[s:q^*\Th(\gr)=q^*\IOut(\ct)\times_{\IOut_{I_0}(\ct)}\Th(\gr,I_0)\simeq B\times \Th(\gr,I_0) \to \Th(\gr,I_0).\]
Thus we can obtain the characteristic map
\[H_{DR}^*(\Th(\gr,I_0))^{\IOut(\gr)}\overset{s^*}\to H^*_{DR}(q^*\Th(\gr))^{\IOut(\pi)}\overset{\Phi_E}\to H_{DR}^*(B).\]
where $\IOut(\pi):=\Ker (\Out(\pi)\to \GL(H))$. This map does not depend on the choice of a fiberwise metric and a trivialization. (Independency of a trivialization of $q^*\IOut(\ct)\to B$ comes from contractiblity of $\IOut_{I_0}(\ct)$.)
\begin{rem}
We remark that the map $\Phi_E$ of Theorem \ref{thm:map} and Remark \ref{rem:map} factors the natural maps \[H^*_{DR}(\Th(\gr))^{\IOut(\gr)}\to H^*_{DR}(q^*\Th(\gr))^{\IOut(\gr)}\to H^*_{DR}(B)\]
by construction. The map $H^*_{DR}(q^*\Th(\gr))^{\IOut(\gr)}\to H^*_{DR}(B)$ is also regarded as $\Phi_E$.
\end{rem}

\subsection{The fiber which has a homogeneous ideal and a splitting}
If the fundamental group $\gr=\pi_1(X)$ of fiber $X$ satisfies some conditions, we can obtain stronger results. In this subsection we assume that the fundamental group $\gr=\pi_1(X)$ of fiber $X$ has the homogeneous ideal $I_\pi$ and there exists a splitting
\[\Der^+(\ct)=V\oplus \Der^+_{I_\gr}(\ct)\]
as graded $S$-vector spaces. Here a group $S$ is the image of a structure group of $E\to B$ in $\GL(\gr)$. 
\begin{ex}
If $\pi$ is a free group, we have $I_\gr=0$. So $\pi$ clearly satisfies the condition above for any $S\subset \GL(\pi)$. For example, the fundamental group of a manifold whose second Betti number is zero is a free group.
\end{ex}
\begin{ex}
We consider the case of $X=\Sigma_g$, which is an oriented surface of genus $g\geq1$. Then $\pi=\pi_1(X)$ has the homogeneous ideal $I_\pi=(\omega)$. Here we denote the intersection form of $\Sigma_g$ by
\[\omega=\sum_{i=1}^g[X_i,X_{i+g}]\in L_2,\]
where $X_1,\dots,X_{2g}$ is a symplectic basis of $H$. We set $S=\Ima(\Diff_+(\Sigma_g)\to \GL(\gr))=\Sp(H_1(\Sigma_g;\Z),\omega)\subset \Sp(H,\omega)$. Since all finite dimensional representations of $\Sp(H,\omega)$ are completely reducible, $\pi$ satisfies the splitting condition.
\end{ex}

\begin{lem}\label{lem:contractible}
{\it The $\GL(\gr)$-equivariant principal $\IAut_{I_\gr}(\ct)$-bundle $\IAut(\ct)\to \mathcal{I}(\gr)$ is $S$-equivariant trivial.}
\end{lem}
\begin{proof}
Let $t$ be a real number. For any $X\in \Der^k(LH)$, we define the map $w_t:\Der^+(\ct)\to \Der^+(\ct)$ by $w_t(X)=t^kX$. If $X\in \Der^k(LH)$ and $Y\in \Der^l(LH)$, we have
\[[w_t(X),w_t(Y)]=[t^kX,t^lY]=t^{k+l}[X,Y]=w_t([X,Y]).\]
So $w_t$ is a Lie algebra homomorphism. We define $F_t:\IAut(\ct)\to \IAut(\ct)$ by \[F_t(\exp X)=\exp w_t(X).\]
Since $w_t$ is a Lie algebra homomorphism, $F_t$ is a group homomorphism. And $F_t$ is commutative with the action of $\GL(H)$ on $\IAut(\ct)$ since $w_t$ preserving the degree of derivations. Furthermore, since $I_\gr$ is a homogeneous ideal, $F_t$ can be restricted to the $\GL(\gr)$-equivariant map $\IAut_{I_\gr}(\ct)\to\IAut_{I_\gr}(\ct)$ and induce $\mathcal{I}(\gr)\to \mathcal{I}(\gr)$.

We set $\IAut^{\leq k}(\ct):=\IAut(\ct/L_{\geq k+1})$. We also write the corresponding Lie algebras as the same way. Since
\[\IAut(\ct)=\varprojlim_k\IAut^{\leq k}(\ct),\]
\[\mathcal{I}(\gr)=\varprojlim_k\IAut^{\leq k}(\ct)/\IAut^{\leq k}_{I_\gr}(\ct).\]
We set $\mathcal{I}(\gr)_k:=\IAut^{\leq k}(\ct)/\IAut_{I_\gr}^{\leq k}(\ct)$. 

By assumption of a splitting
\[\Der^+(\ct)=V\oplus \Der^+_{I_\gr}(\ct),\]
a splitting
\[\Der^{\leq k}(\ct)=V^{\leq k}\oplus \Der^{\leq k}_{I_\gr}(\ct)\]
is induced. Here the superscript $\leq k$ of graded vector spaces means the degree $\leq k$ part.

The map $\Der^{\leq k}(\ct)\to \IAut^{\leq k}(\ct)$ defined by 
\[X+Y\mapsto \exp X\cdot \exp Y\]
is a local diffeomorphism at 0. Therefore so is the map $c:V^{\leq k}\to \mathcal{I}(\gr)_k$ induced by the map above. It gives a diffeomorphism $c:U\to O$ of a neighborhood $U$ of $0$ and a neighborhood $O$ of $[id]$. Then, for all $t>0$ the diagram
\[\xymatrix{U\ar[d]^{w_t}\ar[r]^c&O\ar[d]^{F_t}\\w_t(U)\ar[r]^c&F_t(O)}\]
commutes. Therefore, the restriction $c:w_t(U)\to F_t(U)$ is also a diffeomorphism. Thus
\[c:V^{\leq k}=\bigcup_tw_t(U)\to \bigcup_tF_t(O)=\mathcal{I}(\gr)_k\]
is a global diffeomorphism. The inverse limit 
\[V\to \mathcal{I}(\gr)\]
of the maps is also a diffeomorphism. We obtain a section
\[\mathcal{I}(\gr)\simeq V\overset{\exp}{\to} \IAut(\ct)\]
of the $S$-equivariant principal $\IAut_{I_\gr}(\ct)$-bundle $\IAut(\ct)\to\mathcal{I}(\gr)$.
\end{proof}

Let $\tilde{S}$ be the image of a structure group of $E\to B$ in $\Out(\gr)$. From Lemma \ref{lem:contractible}, the $\GL(\gr)$-equivariant principal $\IOut_{I_\gr}(\ct)$-bundle $\IOut(\ct)\to \mathcal{I}(\gr)$ is also $S$-equivariant trivial. Then there exists a $\tilde{S}$-equivariant diffeomorphism
\begin{multline*}\Th(\gr)=\IOut(\ct)\times_{\IOut_{I_\gr}(\ct)}\Th(\gr,I_\gr)\\
\simeq (\mathcal{I}(\gr)\times \IOut_{I_\gr}(\ct))\times_{\IOut_{I_\gr}(\ct)}\Th(\gr,I_\gr)=\mathcal{I}(\gr)\times \Th(\gr,I_\gr).\end{multline*}
Since the map $F_t:\mathcal{I}(\gr)\to \mathcal{I}(\gr)$ defined in the proof of the theorem above is a $\GL(\gr)$-equivariant homotopy, $\mathcal{I}(\gr)$ is $\GL(\gr)$-contractible. Thus we have the following theorem.
\begin{thm}{\it The space $\Th(\gr)$ is $\tilde{S}$-equivariant homotopic to $\Th(\gr,I_\gr)$.}\end{thm}
Thus the characteristic map $\Phi_E$ of $E\to B$ is described by
\[H^*_{DR}(\Th(\gr,I_\gr))^{\tilde{S}}=H^*_{DR}(\Th(\gr))^{\tilde{S}}\to H^*_{DR}(B).\]

\subsection{Relation between two constructions}

Let $\gr$ be a finitely generated group which has a homogeneous ideal $I_\gr$. We consider the case of $S=1$, $\tilde{S}=\IOut(\gr)$, $I_0=I_\gr$ in Subsection 3.1 and 3.2. We choose a trivialization of $q^*\IOut(\ct)\to B$ which induced by a trivialization of $\IOut(\ct)\to\mathcal{I}(\gr)$ in 2.2. Then the diagram
\[\xymatrix{q^*\Th(\gr)\ar[d]\ar[r]^{\sim\quad\quad} &B\times\Th(\gr,I_\gr)\ar[d]\ar[r]&\Th(\gr,I_\gr)\\
\Th(\gr)\ar[r]^{\sim\quad\quad}&\mathcal{I}(\gr)\times \Th(\gr,I_\gr)\ar[ur]_{\text{h.e.}}&}\]
commutes. So taking $\IOut(\gr)$-invariant de Rham cohomologies, we obtain the diagram
\[\xymatrix{H^*_{DR}(\Th(\gr,I_\gr))^{\IOut(\gr)}\ar[r]\ar[dr]&H^*_{DR}(q^*\Th(\gr))^{\IOut(\gr)}
\ar[r]&H^*_{DR}(B)\\
&H_{DR}^*(\Th(\gr))^{\IOut(\gr)}\ar[u]\ar[ur]&}.\]
Thus the characteristic maps obtained in two ways are equal in common case.

\begin{rem}
In construction of this section, we can obtain characteristic maps of a fiber bundle whose structure group is a subgroup of $\Diff_+(X,*)$ by replacing properly $\Out$, $\Th$, $\ODer^+$,... to $\Aut$, $\Theta$, $\Der^+$,....
\end{rem}

\subsection{Lie algebra cohomology of derivations}
We construct invariant differential forms on $\Th(\pi,I)$ from a flat connection on $\Th(\pi,I)$ and Chevalley-Eilenberg cochains.  

Since the group $\ia$ acts on $\Th(\gr,I)$ freely and transitively, we get the $\de$-coefficient differential form $\eta \in A^1 (\Th (\gr ,I); \de )$ by pulling back the right-invariant Maurer-Cartan form of $\ia$. The form $\eta$ satisfies the equation
\[d\eta+\frac12 [\eta,\eta]=0.\]

\begin{lem}\label{lem:flat}{\it Let $\varphi \in \Out(\pi)$ be an automorphism of $\gr$ which satisfies $|\varphi|(I)=I$. The Maurer-Cartan form $\eta$ on $\Th(\gr,I)$ satisfies 
\[\varphi^*\eta=\Ad(|\varphi|)\eta.\]} \end{lem}
\begin{proof}
For $\theta_0\in \Th(\gr,I)$, we denote $q_{\theta_0}:\Th(\pi,I)\to\ia$ by $\theta\mapsto \theta\circ\theta_0^{-1}$ and the Maurer-Cartan form on $\ia$ by \[\tilde{\eta}\in A^*(\ia;\de).\] Then $\eta=q^*_{\theta_0}\tilde{\eta}$ and it is independent of a choice of $\theta_0$. Since $q_{\theta_0}\circ\varphi=\Ad(|\varphi|)\circ q_{\varphi\cdot\theta_0}$, 
\[\varphi^*\eta=(q_{\theta_0}\circ\varphi)^*\tilde{\eta}=q_{\varphi\cdot\theta_0}^*
\Ad(|\varphi|)^*\tilde{\eta}=\Ad(|\varphi|) q_{\varphi\cdot\theta_0}^*\tilde{\eta}=\Ad(|\varphi|)\eta.\]
Thus we complete the proof.
\end{proof}
From Lemma \ref{lem:flat}, the map $C^*(\de)\to A^*(\Th(\gr,I))^{\IOut(\gr)}$ can be defined by
\[c\mapsto c(\eta^m)\]
for any $c\in C^m(\de)^{\GL(\gr)}$. Here the power $\eta^m$ of $\eta$ is defined by the product 
\begin{multline*}(A^*(X)\otimes\de)^{\otimes m}\\
= A^*(X)^{\otimes m}\otimes \de^{\otimes m}\to A^*(X)\otimes \de^{\otimes m}\end{multline*}
which is the wedge product with respect to $A^*(X)$-components. We set
\[\eta=\sum\eta_\lambda D_\lambda\]
using  a (topological) basis $\{D_\lambda\}_{\lambda\in \Lambda}$ of $\de$ and a well-ordered set $\Lambda$. Then we can write 
\[\eta^m=\sum_{\lambda_1<\cdots<\lambda_m}\eta_{\lambda_1}\wedge \cdots \wedge \eta_{\lambda_m}D_{\lambda_1}\wedge \cdots\wedge D_{\lambda_m}\in A^*(X)\otimes \bigwedge^m \de.\] Thus we can define the above linear map $C^*(\de)\to A^*(\Th(\gr,I))^{\IOut(\gr)}$. This map is a chain map. In fact, we have
\begin{align*}
d(c(\eta^m))&=\sum_{s=1}^m(-1)^{s-1} c(\underbrace{\eta\cdots\overset{s}{d\eta}\cdots\eta}_m)\\
&=\sum_{s=1}^m\frac{(-1)^{s}}{2} c(\underbrace{\eta\cdots\overset{s}{[\eta,\eta]}\cdots\eta}_m)\\
&=(dc)(\eta^{m+1}).\end{align*}
Combining it with the result in Subsection 3.1, we obtain the following:
\begin{thm}
{\it Let $E\to B$ be a fiber bundle whose fiber is $X$ and structure group is $\mathcal{T}(X)$. Fix a decomposable ideal satisfying $\Th(\gr,I_0)\neq\emptyset$. Then the map
\[H^*(\ODer^+(\ct/I_0))\to H^*_{DR}(\Th(\gr,I_0))^{\IOut(\gr)}\to H^*_{DR}(B)\]
is an invariant of an oriented fiber bundle $E\to B$.}
\end{thm}

If $I_0=I_\gr$ is the homogeneous ideal, the group $\GL(\gr)$ acts on $\ODer^+(\ct/I_\gr)$ by the adjoint action. So we define by $C^*(\ODer^+(\ct/I_\gr))^{\GL(\gr)}$ the chain complex of $\GL(\gr)$-invariant Chevalley-Eilenberg cochains of $\ODer^+(\ct/I_\gr)$ with respect to the action. In the same way as the construction above, we obtain the chain map $C^*(\ODer^+(\ct/I_\gr))^{\GL(\gr)}\to A^*(\Th(\gr,I_\gr))^{\Aut(\gr)}$. Since we have the following theorem using the construction in Subsection 3.2.
\begin{thm}
{\it Let $E\to B$ be a fiber bundle whose fiber is $X$ and structure group is a subgroup $\mathcal{G}$ of the diffeomorphism group $\Diff_+(X)$. We set $\pi=\pi_1(X)$, $\tilde{S}=\Ima(\mathcal{G}\to \Out(\gr))$, and $S=\Ima(\mathcal{G}\to \GL(\gr))$. We suppose the group $\gr$ has the homogeneous ideal $I_\pi$ and there exists a splitting
\[\Der^+(\ct)=V\oplus \Der^+_{I_\gr}(\ct)\]
as graded $S$-vector spaces. Then the map
\[H^*(\ODer^+(\ct/I_\gr))^{S}\to H^*_{DR}(\Th(\gr,I_\gr))^{\tilde{S}}\to H^*_{DR}(B)\]
is an invariant of an oriented fiber bundle $E\to B$.}
\end{thm}

\section{Closed surface bundles}

Let $\mathcal{M}_{g}$ be the mapping class group of an oriented closed surface $\Sigma_g$ of genus $g\geq2$ and let $\mathcal{T}_{g}$ be the Teichm\"uler space of genus $g$, which is the space ${\rm Met}_g$ of metrics which has constant curvature $-1$ on $\Sigma_g$ modulo ${\rm Diff}_0(\Sigma_g)$. The quotient orbifold $\mathbb{M}_{g}:=\mathcal{T}_{g}/\mathcal{M}_{g}$ is called the moduli space of genus $g$. Since $\mathcal{T}_{g}$ is contractible, we have the canonical isomorphism
\[H^*(\mathbb{M}_{g};\R)\simeq H^*(\mathcal{M}_{g};\R).\]
We denote the first real homology group $H_1(\Sigma_g;\R)$ of $\Sigma_g$ by $H$ and the intersection form on $\Sigma_g$ by $\omega$. The alternating form $\omega$ can be regarded as an element of $L_2$ by the identity
\[\omega=\sum_{i=1}^g[X_i,X_{i+g}]\]
where $X_1,\dots,X_{2g}$ is any symplectic basis of $H$.

\begin{lem}[\cite{Mor2}]
{\it There exist identifications}
\[\xymatrix{0\ar[r]&\IDer^1(LH/(\omega))\ar@{=}[d]\ar[r]&\Der^1(LH/(\omega))\ar@{=}[d]\ar[r]&\ODer^1(LH/(\omega))\ar@{=}[d]\ar[r]&0\\
0\ar[r]&H\ar[r]^{\cdot\wedge\omega}&\Lambda^3H\ar[r]&\Lambda^3H/H\ar[r]&0.}\]
\end{lem}

According to \cite{KM}, all Morita-Miller-Mumford classes of $\mathcal{M}_{g}$ are constructed from the twisted cohomology class $[\tilde{k}]\in H^1(\mathcal{M}_{g};\Lambda^3 H/H)$, which is the cohomology class of a crossed homomorphism $\tilde{k}:\mathcal{M}_g\to \Lambda^3H/H$ defined in \cite{Mor1}, through the map $\alpha^\text{KM}:\Hom(\Lambda^*(\Lambda^3H/H),\R)^{\Sp(H)}\to H^*(\mathcal{M}_{g};\R)$ defined by
\[c\mapsto c([\tilde{k}]^m)\]
for $c\in \Hom(\Lambda^m(\Lambda^3H/H),\R)^{\Sp(H)}$. Here the power is related to the cup product. We shall describe a relation between the map $\alpha_\text{KM}$ and our construction.

Since every hyperbolic metric on $\Sigma_g$ admits a K\"ahler structure, the corresponding ideal of a hyperbolic metric is $(\omega)$. So we can define the map $\theta_C:\mathcal{T}_{g}\to \Th(\gr,(\omega))$ by giving Chen expansions corresponding to hyperbolic metrics. It can be constructed by the argument which is an analogue for the case of 1-punctured surfaces in \cite{Kaw2}. Pulling back the flat connection $\eta\in A^1(\Th(\gr,(\omega));\ODer^+(\ct/(\omega)))$ defined by the action of $\IOut(\ct/(\omega))$ on $\Th(\gr,(\omega))$, we obtain the flat connection $\theta_C^*\eta\in A^1(\mathbb{M}_{g};\mathcal{T}_{g}\times_{\mathcal{M}_{g}}\ODer^+(\ct/(\omega)))$. Since
\[\theta^*_C:A^*(\Th(\gr,(\omega)))^{\mathcal{M}_{g}}\to A^*(\mathcal{T}_{g})^{\mathcal{M}_{g}}=A^*(\mathbb{M}_{g}),\]
we obtain 
\[\theta_C^*(c(\eta_1^m))=c([\theta^*_C\eta_1]^m)=
(-1)^mc([\tilde{k}]^m)=(-1)^m\alpha^\text{KM}(c)=:\bar{\alpha}^\text{KM}(c)\]
for $c\in \Hom(\Lambda^*(\Lambda^3H/H),\R)^{\Sp(H)}\subset Z^*(\ODer^+(\ct/(\omega)))^{\Sp(H)}$. Here we use $-[\theta^*_C\eta_1]=[\tilde{k}]\in H^1(\mathcal{M}_{g};\Lambda^3H/H)\simeq H^1(\mathbb{M}_{g};\mathcal{T}_{g}\times_{\mathcal{M}_{g}}\Lambda^3H/H)$, where $\eta_1$ is the $\ODer^1(\ct/(\omega))$-component of $\eta$. It is proved as follows: the crossed homomorphism $\tau^\theta_1$ for any $\theta\in \Th(\gr,(\omega))$ is an extension of the first Johnson homomorphism $\tau_1:\mathcal{I}_g\to \Lambda^3H/H$ from Prop \ref{prop:out}. The equation \[[\tilde{k}]=[\tau^\theta_1]\in H^1(\mathcal{M}_{g};\Lambda^3H/H)\] 
follows from the result of \cite{Mor1}. (See \cite{Kaw1}.) The relation between the holonomy $\tau^\theta$ and the iterated integral of the flat connection $\eta$ gives
\[\tau^\theta(\varphi)^{-1}
=\sum_{n=0}^\infty\int_{\theta}^{\varphi\cdot\theta}\underbrace{\eta\cdots\eta}_n\]
for $\varphi\in\mathcal{M}_{g}$ and $\theta\in\Th(\gr,(\omega))$. Therefore the correspondence follows from the equation \[\tau^{\theta_C(x)}_1(\varphi)=-\int_{\theta_C(x)}^{\varphi\cdot \theta_C(x)}\eta_1=-\int_x^{\varphi_* x}\theta_C^*\eta_1\]
for $\varphi\in\mathcal{M}_{g}$ and a fixed point $x\in\mathcal{T}_{g}$. So we obtain $[\tilde{k}]=[\tau^{\theta_C(x)}_1]=-[\theta_C^*\eta_1].$

Thus we get the following theorem:
\begin{thm}\label{thm:coh}
{\it The diagram
\[
\xymatrix{\Hom(\Lambda^*(\Lambda^3H/H),\R)^{\Sp(H)}\ar[drr]_{\bar{\alpha}^\text{KM}}
\ar[r]
&H^*(\ODer^+(\ct/(\omega)))^{\Sp(H)}\ar[r]^{\quad\quad\quad\theta_C^*}&H^*_{DR}(\mathbb{M}_{g})\ar@{=}[d]
\\&&H^*(\mathcal{M}_{g};\R)}
\]
commutes.}
\end{thm}

The map $\theta_C^*$ can be interpreted from the viewpoint of our characteristic map as follows. We consider the oriented $\Sigma_g$-bundle ${\rm Met}_g\times_{{\rm Diff}_+(\Sigma_g)}\Sigma_g\to \mathbb{M}_{g}$. We note that ${\rm Met}_g\times_{{\rm Diff}_+(\Sigma_g)}\Sigma_g\to \mathbb{M}_{g}$ is not a fiber bundle exactly since the fiber on $[x]\in \mathbb{M}_g$, $x\in {\rm Met}_g$, is isomorphic to the grobal quotient orbifold $\Sigma_g/\Isom(\Sigma_g,x)$, where $\Isom(\Sigma_g,x)$ is the isometry group of Riemann surface $(\Sigma_g,x)$. However our construction of the characteristic map works as well as fiber bundles. We give the tautological metric $\mu$ to this bundle, i.e. the metric $\mu_c$ on fiber of $c$ represents the class $c$ for all $c\in\mathbb{M}_{g}$. 

The chain map constructed from the Chen expansion of $\mu$ in the manner of Subsection 3.2 is equal to $\theta_C^*:A^*(\Th(\gr,(\omega)))^{\mathcal{M}_{g}}\to A^*(\mathbb{M}_{g})$. 
\begin{thm}
{\it The homomorphism $\theta_C^*:H^*(\ODer^+(\ct/(\omega)))^{\Sp(H)}\to H^*_{DR}(\mathbb{M}_{g})$ is the characteristic map of the fiber bundle ${\rm Met}_{g}\times_{{\rm Diff}_+(\Sigma_g)}\Sigma_g\to \mathbb{M}_{g}$ constructed in Subsection 3.2.}
\end{thm}
The theorem above means our characteristic maps are non-trivial.

\end{document}